\documentclass{article}
\usepackage{amsfonts}
\usepackage{amsmath}

\newtheorem{theorem}{Theorem}

\newtheorem{conclusion}[theorem]{Conclusion}

\newenvironment{proof}[1][Proof]{\noindent\textbf{#1.} }{\ \rule{0.5em}{0.5em}}
\input{tcilatex}

\begin{document}

\author{Musa Demirci, G\"{o}khan Soydan, Ismail Naci Cang\"{u}l}
\title{RATIONAL POINTS ON ELLIPTIC CURVES $y^{2}=x^{3}+a^{3}\,$IN $\mathbb{F}%
_{p}\mathbb{\,}\,$WHERE $p\equiv 1\,(\limfunc{mod}\,6)$ IS PRIME\thanks{%
This work was supported by the research fund of Uludag University project
no: F-2003/63}}
\date{}
\maketitle

\begin{abstract}
In this work, we consider the rational points on elliptic curves over finite
fields $\mathbb{F}_{p}$. We give results concerning the number of points on
the elliptic curve $y^{2}\equiv x^{3}+a^{3}(\limfunc{mod}\,p)\,$where $p$ is
a prime congruent to $1$ modulo $6$. Also some results are given on the sum
of abscissae of these points. We give the number of solutions to $%
y^{2}\equiv x^{3}+a^{3}(\limfunc{mod}\,p)$, also given in ($\cite{Kob}$,
p.174), this time by means of the quadratic residue character, in a
different way, by using the cubic residue character. Using the Weil
conjecture, one can generalize the results concerning the number of points
in $\mathbb{F}_{p}$ to $\mathbb{F}_{p^{r}}$.
\end{abstract}

\section{\protect\large Introduction}

$\footnote{\textit{AMS 2000 Subject Classification Number : }11G20, 14H25,
14K15, 14G99
\par
\textit{Keywords: Elliptic curves over finite fields, rational points}}$ Let 
$\mathbb{F}$ be a field of characteristic not equal to $2$ or $3$. An
elliptic curve $E$ defined over $\mathbb{F}$ is given by an equation 
\begin{equation}
y^{2}=x^{3}+Ax+B\in \mathbb{F}[x]
\end{equation}%
where $A,\smallskip B\in \mathbb{F}\,\,$so that $4A^{3}+27B^{2}\neq 0$
\thinspace in $\mathbb{F}$. The set of all solutions $(x,y)\in \mathbb{F}%
\times \mathbb{F}$ to this equation together with a point $\,\circ $, called
the point at infinity, is denoted by $E(\mathbb{F})$, and called the set of $%
\mathbb{F}$-rational points on $E.\,$The value$\,\Delta
(E)=-16(4A^{3}+27B^{2})\,$is called\thinspace the\thinspace
discriminant\thinspace of the elliptic curve $E$. For a more detailed
information about elliptic curves in general, see $\cite{Sil}$.

The $E(\mathbb{F})\,$forms\thinspace an\thinspace additive abelian
group\thinspace having\thinspace \thinspace identity $\circ $. Here by
definition, $-P=(x,-y)\,$for a\thinspace point $P=(x,y)\,$on$\,E.$

It has always been interesting to look for the number of points over a given
field $\mathbb{F}.\,$In $\cite{Sch}$, three algorithms to find the number of
points on an elliptic curve over a finite field are given.

\section{{\protect\large The Group E(}$F_{p}${\protect\large )\thinspace
of\thinspace \thinspace Points Modulo \thinspace p,\thinspace p }$\equiv $ 
{\protect\large 1\thinspace\ (}$\limfunc{mod}$ {\protect\large 6)\thinspace
\thinspace }}

It is interesting to solve polynomial\thinspace congruences modulo $p.\,$%
Clearly, it is much easier to find solutions in $\mathbb{F}_{p}$ for small $%
p $, than to find them in $\mathbb{Q}.$ Because, in $\mathbb{F}_{p},$ there
is always a finite number of solutions.

Let $\alpha \in \mathbb{F}_{p}$ and let $p$ be as stated earlier, then the
number of solutions to $x^{3}=\alpha $ is given by $1+\chi _{3}(\alpha
)+\chi _{3}^{2}(\alpha )$ for a cubic character $\chi _{3}$ (so $\chi _{3}:%
\mathbb{F}_{p}^{\ast }\rightarrow \{1,\omega ,\omega ^{2}\}$ where $\omega $
is a non-trivial cubic root of unity). Likewise, let $\chi (a)=(\frac{a}{p})$
denote the Legendre symbol which is equal to $+1\,$if $a$ $\,$is a quadratic
residue modulo $p$; $-1\,$if not; and 0 if $p|a$, ($\cite{Sil}$, p.132).$~$%
The number of solutions to $x^{2}=\alpha $ is then $1+\chi (\alpha ).$

In this work, we consider the elliptic curve $(1)$ in modulo $p$, for $A=0$
and $B=a^{3}$, and denote it by $E_{a}$. We try to obtain results concerning
the number of points on $E_{a}$\thinspace over $\mathbb{F}_{p}$, and also
their orders.

Let us denote the set of $\mathbb{F}_{p}$-rational points on $E_{a}$ by $%
E_{a}(\mathbb{F}_{p})$, and let $N_{p,a}~$be the cardinality of the set $%
E_{a}(\mathbb{F}_{p}).$ It is known that the number of solutions of $y^{2}=u$
$(\func{mod}$ $p)$ is $1+\chi (u)$, and so the number of solutions to $%
y^{2}\equiv x^{3}+a^{3}\,(\limfunc{mod}$ $p)$, counting the point at
infinity, is%
\begin{eqnarray*}
N_{p,a} &=&1+\underset{x\in \mathbb{F}_{p}}{\sum }(1+\chi (x^{3}+a^{3})) \\
&=&p+1+\underset{x\in \mathbb{F}_{p}}{\sum }\chi (x^{3}+a^{3}).
\end{eqnarray*}

It can easily be seen that an elliptic curve 
\begin{equation}
y^{2}=x^{3}+a^{3}
\end{equation}%
can have at most $2p+1$ points in $\mathbb{Z}_{p}$; i.e. the point at
infinity along with $2p$ pairs $(x,y)$ with $x,y\in \mathbb{F}_{p}$,
satisfying the equation $(2)$. This is because, for each $\,x\in \mathbb{F}%
_{p}$, \thinspace there are at most two possible values of $y\in \mathbb{F}%
_{p},$ satisfying $(2)$.

But not all elements of $\mathbb{F}_{p}$ have a square root. In fact, only
half of the elements in $\mathbb{F}_{p}^{\,\ast \,}=\mathbb{F}_{p}\backslash
\{\overline{0}\}$ have square roots. Therefore the $_{{}}$expected number of
points on $E(\mathbb{F}_{p})$ is about $p+1.$

It is known, as a more precise formula, that the number of solutions to $(2)$
is 
\begin{equation*}
p+1+\sum \chi (x^{3}+a^{3}).
\end{equation*}
The following theorem of Hasse quantifies this result:

\begin{theorem}
(Hasse, 1922) An elliptic curve $(2)$ has 
\begin{equation*}
p+1+\delta
\end{equation*}%
solutions $(x,y)$ modulo $p$, where $|\delta |$ $<2\sqrt{p}.$

Equivalently, the number of solutions is bounded above by the number $(\sqrt{%
p}+1)^{2}.$
\end{theorem}

From now on, we will only consider the case where $p$ is a prime congruent
to $1$ modulo $6$. We begin by some calculations regarding the number of
points on $(2).$ First we have

\begin{theorem}
Let $p\equiv 1$ $(\limfunc{mod}$ $6)$ be a prime. The number of points $%
(x,y) $ on the curve $y^{2}=x^{3}+a^{3}\,$modulo $p$ is given by 
\begin{equation*}
4+\underset{x\in \mathbb{F}_{p}}{\sum }\rho (x)
\end{equation*}%
where%
\begin{equation*}
\rho (x)=\left\{ 
\begin{array}{cc}
2 & if~\chi (x^{3}+a^{3})=1 \\ 
0 & if~\chi (x^{3}+a^{3})\neq 1%
\end{array}%
\right.
\end{equation*}%
Also the sum of such $y$ is $p.$
\end{theorem}

\begin{proof}
For $x=0,1,2,...,p-1\,(\limfunc{mod}\,p)$, find the values $%
y^{2}=x^{3}+a^{3}\,(\limfunc{mod}$ $p).$ Let$\,\ Q_{p}$ denote the set of
quadratic residues modulo $p$. When$\ \ y^{2}\in Q_{p},$ then there are two
values of\thinspace $\ y\in U_{p},$ the set of units in $\mathbb{F}_{p}$;
which are $x_{0}\,$and $p-x_{0}$. When $y=0$, there are three more points
which are $x=a,$ $x=wa~$and $x=w^{2}a$ where $w^{2}+w+1=0.$ (Here $w\in 
\mathbb{F}_{p}$ since $p\equiv 1$ $(\limfunc{mod}$ $6)$). Finally
considering the point at infinity, the result follows.
\end{proof}

We now consider the points on $(2)$ $\,$lying on the $y$-axis.

\begin{theorem}
Let $p\equiv 1$ $(\limfunc{mod}$ $6)$ be prime. For $x\equiv 0$ $\,(\limfunc{%
mod}$ $p),$ there are two points on the curve $y^{2}\equiv x^{3}+a^{3}\,(%
\limfunc{mod}$ $p)$, when $a\in Q_{p}$, while when $a\notin Q_{p}$, there is
no point with $x\equiv 0$ $(\func{mod}$ $p)$.
\end{theorem}

\begin{proof}
For $x\equiv 0\,(\limfunc{mod}p)$, we have $y^{2}\equiv a^{3}\,(\limfunc{mod}%
\,p)$. First consider $y^{2}\equiv a^{3}\,(\limfunc{mod}\,p).$ This
congruence has a solution if and only if $\left( \frac{a^{3}}{p}\right)
=\left( \frac{a}{p}\right) =1$; i.e. if and only if $a$ is a quadratic
residue modulo $p.$
\end{proof}

Let us now denote by $K_{p},\,$the set of cubic residues modulo $p$. We can
now restate the result given just before\thinspace Hasse's theorem in terms
of cubic residues modulo $p$, instead of quadratic residues.

\begin{theorem}
Let $p\equiv 1$ $(\limfunc{mod}$ $6)\,$be prime. Let $t=y^{2}-a^{3}$. Then
the number of points on the curve $y^{2}\equiv x^{3}+a^{3}\,(\limfunc{mod}$ $%
p)$ is given by the sum 
\begin{equation*}
1+\sum f(t)
\end{equation*}%
where 
\begin{equation*}
f(t)=\left\{ 
\begin{array}{cc}
0 & \text{if\thinspace \thinspace }\,\,t\notin K_{p}, \\ 
1 & \text{if\thinspace \thinspace }\,\,p|\,t, \\ 
3 & \text{if\thinspace \thinspace }\,t\in K_{p}^{\ast },%
\end{array}%
\right.
\end{equation*}%
and the sum is taken over all $y\in \mathbb{F}_{p}$.
\end{theorem}

\begin{proof}
Let $p|t$. Then the equation $x^{3}\equiv t\,$\ $(\limfunc{mod}\,p)$ becomes$%
\,\ x^{3}\equiv 0\,\ (\limfunc{mod}\,p).$ Then the unique solution is $%
x\equiv 0$ $(\limfunc{mod}\,p).$ Therefore $f(t)=1.$

Let secondly $t\notin K_{p}.$ Then\thinspace \thinspace $t\,\ $is not a
cubic residue and the congruence $x^{3}\equiv t\,(\limfunc{mod}\,p)\,$ has
no solutions. If $t\in K_{p}^{\ast },$ then $x^{3}\equiv t\,$\ $(\limfunc{mod%
}\,p)\,\,$ has three solutions since $p\equiv 1$ $(\limfunc{mod}$ $6)\,$ and 
$(p-1,3)=3.$
\end{proof}

We can also give a result about the sum of abscissae of the rational points
on the curve:

\begin{theorem}
Let $p\equiv 1$ $(\limfunc{mod}$ $6)\,$be prime. The sum of abscissae of the
rational points on the curve $y^{2}\equiv x^{3}+a^{3}\,$\ $(\limfunc{mod}$ $%
p)\,$is 
\begin{equation*}
\underset{x\in \mathbb{F}_{p}}{\sum }(1+\chi _{p}(x^{3}+a^{3})).x\text{.}
\end{equation*}
\end{theorem}

\begin{proof}
Since 
\begin{equation*}
\chi _{p}(t)=\left\{ 
\begin{array}{cc}
+1 & \text{if }x^{2}\equiv t\text{ }(\limfunc{mod}\,p)\,\text{has a solution,%
} \\ 
0 & \text{if \thinspace }p|t\text{,} \\ 
-1 & \text{if }x^{2}\equiv t\text{ }(\limfunc{mod}\,p)\,\text{has no
solutions,}%
\end{array}%
\right.
\end{equation*}%
we know that $1+\chi _{p}(t)=0,1\,$or $2$. When $y\equiv 0\,(\limfunc{mod}%
\,p),\,x^{3}+a^{3}\equiv 0\,(\limfunc{mod}\,p)\,$and hence as $p|0,\,\chi
_{p}(x^{3}+a^{3})=0.$ For each such point $(x,0)$ on the curve, $(1+0).x=x\;$%
is added to the sum.

Let $x^{3}+a^{3}=t.\,$If $(\frac{t}{p})=+1$, then for each such point $%
(x,y)\,$on the curve, the point $(x,-y)\,$is also on the curve. Therefore
for each such $t$, $(1+1).x=2x\,$is added to the sum.

Finally if $(\frac{t}{p})=-1,$ then the congruence $x^{2}\equiv t$ $(%
\limfunc{mod}$ $p)$ has no solutions, and such points $(x,y)$ contribute to
the sum as much as $(1+(-1)).x=0$.
\end{proof}

As we can see from the following result, the above sum is always congruent
to $0$ modulo p:

\begin{theorem}
Let $p\equiv 1$ $(\limfunc{mod}$ $6)~$be prime. Then the sum of the integer
solutions of $x^{3}\equiv t\,(\limfunc{mod}$ $p)$ is congruent to $0$ modulo 
$p.$
\end{theorem}

\begin{proof}
The solutions of the congruence $x^{3}\equiv 1\,(\limfunc{mod}$ $p)$ are $%
x\equiv 1,$ $w$ and $w^{2}\,(\limfunc{mod}\,p)$, where $w=\frac{-1+\sqrt{3}i%
}{2}\,$is the cubic root of unity. In general, the solutions of $x^{3}\equiv
t\,(\limfunc{mod}\,p)$ are $x\equiv x_{0},$ $x_{0}w\,$and $x_{0}w^{2}(%
\limfunc{mod}$ $p),$ where $x_{0}$ is a solution. This is because $%
(x_{0}w)^{3}\equiv x_{0}^{3}w^{3}\equiv x_{0}^{3}\equiv t\,(\limfunc{mod}%
\,p) $ and similarly $(x_{0}w^{2})^{3}\equiv x_{0}^{3}w^{6}\equiv
x_{0}^{3}(w^{3})^{2}\equiv x_{0}^{3}\equiv t\,\,(\limfunc{mod}\,p).$
Therefore the sum of these solutions is 
\begin{equation*}
x_{0}+x_{0}w+x_{0}w^{2}=x_{0}+x_{0}w+x_{0}(-1-w)=0
\end{equation*}%
If there is no solution, the sum can be thought of as $0$.
\end{proof}

We can now prove the following:

\begin{theorem}
Let $p\equiv 1$ $(\limfunc{mod}$ $6)~$be prime. Let $0\leq x\leq p-1$ be an
integer. Then for any $1\leq a\leq p-1$, the sum (which is an integer)%
\begin{equation*}
j(p)=\overset{p-1}{\underset{x=0}{\dsum }}(1+\chi (x^{3}+a^{3}))x
\end{equation*}%
is divisible by $p$. In particular%
\begin{equation*}
s(p)=\overset{p-1}{\underset{x=0}{\dsum }}\chi (x^{3}+a^{3})x
\end{equation*}%
is divisible by $p.$
\end{theorem}

\begin{proof}
For every value of y, let $y^{2}-a^{3}=t\,.$ Then the sum of the solutions
of the congruence $x^{3}\equiv t\,$\ $(\limfunc{mod}$ $p)\,$is congruent to $%
0$ by Theorem 6.

For all values of $y$, this is valid and hence the sum of all these
abscissae is congruent to $0.$
\end{proof}

The hypothesis $p\equiv 1$ $(\limfunc{mod}$ $6)$ is essential in this
Theorem 7, as the following counterexample shows: take $a=1,~p=11.$ Then the
first sum is easily seen to be $56~$and the second is easily seen to be $1$
and clearly neither of them is divisible by $11.$

We now look at the points on the curve having the same ordinate:

\begin{theorem}
Let $p\equiv 1$ $(\limfunc{mod}$ $6)$ be prime. The sum of the abscissae of
the points $(x,y)$ on the curve $y^{2}\equiv x^{3}+a^{3}\,(\limfunc{mod}$ $%
p) $ having the same ordinate $y$, is congruent to zero modulo $p$.
\end{theorem}

\begin{proof}
Let $y$ be given. Then the congruence 
\begin{equation*}
x^{3}\equiv y^{2}-a^{3}\text{ }(\limfunc{mod}\,p)
\end{equation*}%
becomes 
\begin{equation*}
x^{3}\equiv t\,\text{\ }(\limfunc{mod}\,p)\,
\end{equation*}%
after a substitution $t=y^{2}-a^{3}$. The result then follows by Theorem 6.

Finally we consider the total number of points on a family of curves $%
y^{2}\equiv x^{3}+a^{3}$ $(\limfunc{mod}$ $p),\,$ for $a\equiv 0,1,...,$ $%
p-1\,$\ $(\limfunc{mod}\,p)\,$ and $p\equiv 1$ $(\limfunc{mod}$ $6)$ is
prime. We find that when $(a,p)=1,$ there are $p+1-2k$ or $p+1+2k$ points on
a curve $y^{2}\equiv x^{3}+a^{3}\,(\limfunc{mod}$ $p)$, for a suitable
integer $k$.
\end{proof}

\begin{theorem}
Let $p\equiv 1$ $(\limfunc{mod}$ $6)$ be prime and let $1\leq a\leq p-1$.
Let $N_{p,a}=\#E(\mathbb{F}_{p}).$ Then%
\begin{equation*}
\overset{p-1}{\underset{a=1}{\dsum }}N_{p,a}=p^{2}-1.
\end{equation*}
\end{theorem}

\begin{proof}
Since $1\leq a\leq p-1$, we have $(a,p)=1.$ Then the set of elements $%
a^{3}x^{3}$ modulo $p$ is the same as the set of $x^{3}$ modulo $p$. Then%
\begin{eqnarray*}
\underset{x\in \mathbb{F}_{p}}{\sum }\chi (x^{3}+a^{3}) &=&\underset{x\in 
\mathbb{F}_{p}}{\sum }\chi (a^{3}x^{3}+a^{3}) \\
&=&\chi (a^{3}).\underset{x\in \mathbb{F}_{p}}{\sum }\chi (x^{3}+1)~.
\end{eqnarray*}%
By the discussion at the beginning of section $2,$ we get%
\begin{equation*}
N_{p,a}-p-1=\chi (a^{3}).(N_{p,1}-p-1)\text{ .}
\end{equation*}%
Then by taking sum at both sides, we obtain%
\begin{equation*}
\overset{p-1}{\underset{a=1}{\dsum }}(N_{p,a}-p-1)=\overset{p-1}{\underset{%
a=1}{\dsum }}\chi (a^{3}).(N_{p,1}-p-1)\text{ .}
\end{equation*}%
Then%
\begin{eqnarray*}
\overset{p-1}{\underset{a=1}{\dsum }}N_{p,a}-\overset{p-1}{\underset{a=1}{%
\dsum }}(p+1) &=&(N_{p,1}-p-1).\overset{p-1}{\underset{a=1}{\dsum }}\chi
(a^{3}) \\
&=&(N_{p,1}-p-1).\overset{p-1}{\underset{a=1}{\dsum }}\chi (a)
\end{eqnarray*}%
using $\chi (a^{3})=\chi (a)$ as both sides are $1$ or $-1$ . Finally as
there are as many residues as non residues, we know that 
\begin{equation*}
\overset{p-1}{\underset{a=1}{\dsum }}\chi (a)=0\text{ }
\end{equation*}%
and by means of it, we conclude%
\begin{equation*}
\overset{p-1}{\underset{a=1}{\dsum }}N_{p,a}=p^{2}-1\text{,}
\end{equation*}%
as required.
\end{proof}

\begin{conclusion}
All the results concerning the number of points on $\mathbb{F}_{p\text{ }}$%
obtained here for prime $p\equiv 1$ $(\limfunc{mod}$ $6)$ can be generalized
to $\mathbb{F}_{p^{r}}$, for a natural number $r>1$, using the following
result:
\end{conclusion}

\begin{theorem}
(Weil Conjecture) The Zeta-function is a rational function of $T$ having the
form%
\begin{equation*}
Z(T;E/\mathbb{F}_{q})=\frac{1-aT+qT^{2}}{(1-T)(1-qT)}
\end{equation*}%
where only the integer $a$ depends on the particular elliptic curve $E$. The
value $a$ is related to $N=N_{1}$ as follows:%
\begin{equation*}
N=q+1-a\text{.}
\end{equation*}
\end{theorem}

\bigskip In addition, the discriminant of the quadratic polynomial in the
numerator is negative, and so the quadratic has two conjugate roots $\frac{1%
}{\alpha }$ and $\frac{1}{\beta }$ with absolute value $\frac{1}{\sqrt{q}}$.
Writing the numerator in the form $(1-\alpha T)(1-\beta T)$ and taking the
derivatives of logarithms of both sides, one can obtain the number of $%
F_{q^{r}}$- points on $E,$ denoted by $N_{r}$, as follows:%
\begin{equation*}
N_{r}=q^{r}+1-\alpha ^{r}-\beta ^{r}\text{, \ }r=1,\text{ }2,\text{ }...
\end{equation*}%
\textbf{Example 12 }Let us find the $F_{7}$-points on the elliptic curve $%
y^{2}=x^{3}+4^{3}$. There are $N_{1}=12$ $F_{7}$-points on the elliptic
curve: 
\begin{equation*}
(0,1),~(0,6),~(1,3),~(1,4),~(2,3),~(2,4),~(3,0),~(4,3),~(4,4),~(5,0),~(6,0)\ 
\text{and }\circ .
\end{equation*}%
~ Now as $r=2,$ we have $a=-4.$ Then from the quadratic equation 
\begin{equation*}
1+4T+7T^{2}=0\text{,}
\end{equation*}%
$\alpha =-2-\sqrt{3}i$ and $\beta =-2+\sqrt{3}i$ and finally $N_{2}=48$.
Similarly $N_{3}=324$ can be calculated.

\begin{tabular}{l}
Musa Demirci, G\"{o}khan Soydan, Ismail Naci Cang\"{u}l \\ 
Department of Mathematics \\ 
Uluda\u{g} University \\ 
16059 Bursa, TURKEY \\ 
mdemirci@uludag.edu.tr, gsoydan@uludag.edu.tr, cangul@uludag.edu.tr%
\end{tabular}


\begin{thebibliography}{9}
\bibitem{Kob} Koblitz, N., \emph{A Course in Number Theory and Cryptography}%
, Springer-Verlag, (1994), ISBN 3-540-94293-9.

\bibitem{Mol} Mollin, R. A., \emph{An Introduction to Cryptography},
Chapman\&Hall/CRC, (2001), ISBN 1-58488-127-5.

\bibitem{Sil} Silverman, J. H., \emph{The Arithmetic of Elliptic Curves},
Springer-Verlag, (1986), ISBN 0-387-96203-4.

\bibitem{Silv} Silverman, J. H.,Tate, J., \emph{Rational Points on Elliptic
Curves}, Springer-Verlag, (1992), ISBN 0-387-97825-9.

\bibitem{Sch} Schoof, R., \emph{Counting points on elliptic curves over
finite fields}, Journal de Th\'{e}orie des Nombres de Bordeaux, 7 (1995),
219-254.
\end{thebibliography}
\end{document}